\documentclass[12pt,oneside]{amsart}
\usepackage[utf8]{inputenc}
\usepackage[T1]{fontenc}
\usepackage{amsmath,amssymb,amsthm}
\usepackage[letterpaper,portrait,lmargin=.9in,rmargin=.9in,tmargin=.8in,bmargin=.9in,headsep=0.25in]{geometry}
\usepackage[pdftex,bookmarks,colorlinks,breaklinks]{hyperref}
\hypersetup{linkcolor=blue,citecolor=red,filecolor=magenta,urlcolor=blue,pdftitle={Odd-Girth Bounds for Defective Edge Coloring},pdfauthor={Guantao Chen and Alireza Fiujlaali}}

\newtheorem{theorem}{Theorem}[section]
\newtheorem{corollary}[theorem]{Corollary}
\newtheorem{lemma}[theorem]{Lemma}
\newtheorem{conjecture}[theorem]{Conjecture}
\newtheorem{claim}{Claim}[section]


\begin{document}

\title[Odd-Girth Bounds for Defective Edge Coloring]{Odd-Girth Bounds for Defective Edge Coloring}

\author{Guantao Chen}
\email{gchen@gsu.edu}
\author{Alireza Fiujlaali}
\email{afiujlaali1@gsu.edu}
\address{Department of Mathematics and Statistics, Georgia State University, Atlanta, GA USA 30303-2918}

\begin{abstract}
A $(k,d)$-edge coloring of a loopless multigraph $G$ is an edge coloring using at most $k$ colors such that the subgraph formed by each color class has maximum degree at most $d$. The least such $k$ is denoted by $\chi'_d(G)$. Let $G$ be a loopless non-bipartite multigraph with maximum degree $\Delta(G)$ and odd girth $g_0(G)$, and let $d\ge1$ be odd. We prove that
\[
\chi'_d(G)\le\left\lceil\frac{g_0(G)\Delta(G)-1}{dg_0(G)-1}\right\rceil.
\]
For $d=1$, this is Goldberg's odd-girth refinement of Shannon's theorem, while for $g_0(G)=3$ it is the defective Shannon bound of Aboulker, Aubian, and Huang. For every odd $d>1$, every odd $g_0\ge3$, and every $\Delta>d$, an almost full ring multigraph $R(\Delta,g_0)$, an odd cycle with edge multiplicities alternating between $\lfloor\Delta/2\rfloor$ and $\lceil\Delta/2\rceil$, except that two consecutive edges have multiplicity $\lfloor\Delta/2\rfloor$, attains equality. We also derive a range in which the defective Goldberg--Seymour conjecture holds.
\end{abstract}

\maketitle

\section{Introduction}

For a multigraph $G$, let $\underline G$ denote the \textit{underlying} simple graph obtained by replacing each set of parallel edges by a single edge, and let $\Delta(G)$ denote the maximum degree of $G$. If $G$ contains an odd cycle, its \textit{odd girth}, denoted by $g_0(G)$, is the length of a shortest odd cycle. A \textit{$d$-defective edge coloring} of $G$ is an edge coloring such that the subgraph formed by the edges of each color class has maximum degree at most $d$. We say that $G$ is \textit{$(k,d)$-edge colorable} if it admits such a coloring using at most $k$ colors. The least such $k$ is the \textit{$d$-defective chromatic index} of $G$, denoted by $\chi'_d(G)$. When $d=1$, this is the usual chromatic index, so $\chi'_1(G)=\chi'(G)$.

Defective edge coloring is the constant-function case of $f$-coloring; see Wdowinski~\cite{wdowinski2023} for a broader Goldberg--Seymour framework for $f$-colorings. It is immediate that every loopless multigraph $G$ satisfies $\chi'_d(G)\ge\lceil\Delta(G)/d\rceil$. Hakimi and Kariv~\cite{hakimi1986} proved that equality holds for every loopless multigraph when $d$ is even and for every bipartite multigraph when $d$ is arbitrary.  Thus it remains to consider non-bipartite multigraphs and odd $d$. Every such multigraph has finite odd girth, and our main result gives the following odd-girth bound.

\begin{theorem}\label{main}
Let $G$ be a loopless non-bipartite multigraph, and let $d\ge1$ be odd. Then
\[
\chi'_d(G)\le\left\lceil\frac{g_0(G)\Delta(G)-1}{dg_0(G)-1}\right\rceil.
\]
\end{theorem}

For odd $d>1$ and $\Delta>d$, the bound is sharp and is attained by the \textit{almost full ring graph} $R(\Delta,g_0)$, whose underlying graph is an odd cycle of length $g_0$ and whose edge multiplicities alternate between $\lfloor\Delta/2\rfloor$ and $\lceil\Delta/2\rceil$, except that two consecutive edges have multiplicity $\lfloor\Delta/2\rfloor$. Aboulker, Aubian, and Huang proved the following sharp bound for odd $d$.

\begin{theorem}[Aboulker, Aubian, and Huang~\cite{aboulker2022vizing}; Shannon~\cite{shannon1949} for $d=1$]\label{3french}
Let $G$ be a loopless multigraph with maximum degree $\Delta$. If $d>1$ is odd, then
$
\chi'_d(G)\le\left\lceil\frac{3\Delta-1}{3d-1}\right\rceil,
$
and this bound is sharp for all $\Delta$ and odd $d>1$. For $d=1$, the same inequality is Shannon's theorem.
\end{theorem}

For $d=1$,
$
\lceil(3\Delta(G)-1)/2\rceil=\lfloor3\Delta(G)/2\rfloor.
$
When $d>1$ and $g_0(G)=3$, Theorem~\ref{main} reduces to Theorem~\ref{3french} and refines it when $g_0(G)>3$. Goldberg refined Shannon's theorem by using the odd girth.

\begin{theorem}[Goldberg~\cite{goldberg1972}]\label{gold}
For every non-bipartite graph $G$ with odd girth $g_0(G)$,
\[
\chi'(G)\le\Delta(G)+1+\left\lfloor\frac{\Delta(G)-2}{g_0(G)-1}\right\rfloor.
\]
\end{theorem}

Since $\lceil a/b\rceil=1+\lfloor(a-1)/b\rfloor$ for positive integers $a$ and $b$, we have
\begin{align*}
\Delta(G)+1+\left\lfloor\frac{\Delta(G)-2}{g_0(G)-1}\right\rfloor
=\Delta(G)+\left\lceil\frac{\Delta(G)-1}{g_0(G)-1}\right\rceil=\left\lceil\frac{g_0(G)\Delta(G)-1}{g_0(G)-1}\right\rceil.
\end{align*}
Thus the $d=1$ case of Theorem~\ref{main} is Goldberg's bound.

For ordinary edge coloring, extremal and structural questions associated with Goldberg's odd-girth bound have been studied through ring graphs. Cao, Chen, He, and Jing~\cite{cao2020} investigated when equality in Goldberg's bound forces the presence of an extremal ring subgraph, and Fan, Long, Wang, and Wang~\cite{fan2026} recently characterized when this phenomenon occurs. The present paper develops a defective analogue of Goldberg's bound and identifies almost full ring graphs as sharp examples.

For every $X\subseteq V(G)$ with $|X|>2$, each color class in a $d$-defective edge coloring of $G$ contains at most $\lfloor d|X|/2\rfloor$ edges of $G[X]$. Hence $\chi'_d(G)\ge\Gamma_d(G)$, where
\[
\Gamma_d(G):=\max\left\{\left\lceil\frac{|E(G[X])|}{\lfloor d|X|/2\rfloor}\right\rceil:X\subseteq V(G),\ |X|>2\right\}.
\]
Here $\Gamma_d(G)$ is the \textit{$d$-defective density} of $G$. If $|X|$ is even, then
\[
\left\lceil\frac{|E(G[X])|}{\lfloor d|X|/2\rfloor}\right\rceil\le\left\lceil\frac{\Delta(G)}d\right\rceil\le\left\lceil\frac{\Delta(G)+1}d\right\rceil.
\]
Thus even sets are dominated by the maximum-degree term in Conjecture~\ref{dgs}, and the conjecture is equivalently expressible using only odd sets. When $d=1$, $\Gamma_d(G)$ is the usual density in edge coloring. Goldberg~\cite{goldberg1973} and Seymour~\cite{seymour1979} independently conjectured that every loopless multigraph $G$ satisfies
$
\chi'(G)\le\max\{\Gamma_1(G),\Delta(G)+1\}.
$
This conjecture was proved by Chen, Jing, and Zang~\cite{chen2025goldberg}. Aboulker, Aubian, and Huang proposed the following defective analogue.

\begin{conjecture}[Aboulker, Aubian, and Huang~\cite{aboulker2022vizing}]\label{dgs}
Every loopless multigraph $G$ satisfies
\[
\chi'_d(G)\le\max\left\{\Gamma_d(G),\left\lceil\frac{\Delta(G)+1}{d}\right\rceil\right\}.
\]
\end{conjecture}

Our main result shows that Conjecture~\ref{dgs} holds in the following range. Let $\Delta(G)=md+s$, where $m\ge0$ and $0\le s<d$, and let
$
\ell:=\max\{1,\Gamma_d(G)-m\}.
$
Then Conjecture~\ref{dgs} holds whenever
\[
g_0(G)\ge\left\lceil\frac{m+\ell-1}{\ell d-s}\right\rceil.
\]

The remainder of the paper is organized as follows. Section~\ref{pre} contains the preliminary results, including the construction showing that the bound in Theorem~\ref{main} is sharp. Section~\ref{main-result} proves Theorem~\ref{main} and derives its application to Conjecture~\ref{dgs}.

\section{Preliminaries}\label{pre}

All multigraphs considered in this paper are loopless. Let $V(G)$ and $E(G)$ denote the vertex set and edge set of a multigraph $G$, respectively. For each $uv\in E(\underline G)$, let $\mu_G(uv)$ denote the number of edges joining $u$ and $v$ in $G$. When the graph is fixed, let $\mu(uv):=\mu_G(uv)$. For a subgraph $H$ of $G$ and $v\in V(H)$, let $d_H(v)$ denote the degree of $v$ in $H$. For $A\subseteq V(G)$, let $G[A]$ be the subgraph induced by $A$, and let $G-A:=G[V(G)\setminus A]$. For a positive integer $n$, let $[n]:=\{1,\ldots,n\}$, and let $C_n$ denote the cycle of length $n$. A \textit{$k$-factor} is a spanning $k$-regular subgraph, and a \textit{near-perfect matching} is a matching that misses exactly one vertex. For a non-bipartite multigraph $G$ and a fixed odd integer $d>1$, let
\[
p(G):=\left\lceil\frac{g_0(G)\Delta(G)-1}{dg_0(G)-1}\right\rceil.
\]
When the graph is fixed, let $p:=p(G)$.

For a positive integer $\Delta$ and an odd integer $g_0\ge3$, the \textit{almost full ring graph}, denoted by $R(\Delta,g_0)$, is the multigraph whose underlying graph is the cycle $v_1v_2\cdots v_{g_0}v_1$ and whose edge multiplicities satisfy
\[
\mu_{R(\Delta,g_0)}(v_iv_{i+1})=
\begin{cases}
\lfloor\Delta/2\rfloor,&\text{if }i\text{ is odd},\\
\lceil\Delta/2\rceil,&\text{if }i\text{ is even},
\end{cases}
\]
for $1\le i\le g_0-1$, and $\mu_{R(\Delta,g_0)}(v_{g_0}v_1)=\lfloor\Delta/2\rfloor$. When $\Delta$ is even, $R(\Delta,g_0)$ is $\Delta$-regular. When $\Delta$ is odd, $v_1$ is its unique vertex of degree $\Delta-1$, and every other vertex has degree $\Delta$.

We use Petersen's theorem on $2$-factors and König's theorem on bipartite edge coloring.

\begin{theorem}[Petersen~\cite{petersen1891}]\label{pet}
Every loopless regular multigraph of positive even degree contains a $2$-factor.
\end{theorem}

\begin{theorem}[König~\cite{konig1916}]\label{konig}
Every bipartite multigraph $G$ satisfies $\chi'(G)=\Delta(G)$.
\end{theorem}

The following theorem of Kano provides the even factors needed in odd regular multigraphs.

\begin{theorem}[Kano~\cite{kano1984}]\label{kano}
Let $\Delta$ be odd, and let $G$ be a $2$-edge-connected $\Delta$-regular loopless multigraph. For every edge $e\in E(G)$ and every positive even integer $k\le2\Delta/3$, the graph $G$ has a $k$-factor containing $e$.
\end{theorem}

The following bound helps us obtain a dense odd set when the required proper edge coloring does not exist.
\begin{theorem}[Chen, Jing, and Zang~\cite{chen2025goldberg}]\label{gs}
Every loopless multigraph $G$ satisfies
$
\chi'(G)\le\max\{\Gamma_1(G),\Delta(G)+1\}.
$
\end{theorem}

The following lemma colors regular multigraphs whose degree is slightly below a multiple of $d$.
\begin{lemma}\label{lemma:absorption}
Let $G$ be a $(qd-h)$-regular loopless multigraph with odd girth $g_0(G)$, where $d>1$ is odd and $q,h$ are positive integers. If $h\le q\le hg_0(G)$ and $qd-h$ is even, then
$
\chi'_d(G)\le q.
$
\end{lemma}

\begin{proof}
All unions in this proof are edge-disjoint. Since
$
qd-h=q(d-1)+(q-h)
$
and $d-1$ is even, $q-h$ is even. By repeatedly applying Theorem~\ref{pet}, let $F_1,\ldots,F_q$ be edge-disjoint $(d-1)$-factors of $G$, each a union of $(d-1)/2$ $2$-factors, and let
$
G':=G-\bigcup_{i=1}^qF_i.
$
Then $G'$ is $(q-h)$-regular. If $q=h$, then $G'=\emptyset$, and $F_1,\ldots,F_q$ give a $(q,d)$-edge coloring of $G$. Thus assume that $q>h$.

If $G'$ is bipartite, then Theorem~\ref{konig} gives a proper edge coloring of $G'$ with $q-h$ colors. Let $M_1,\ldots,M_{q-h}$ be its color classes. For every $i\in[q-h]$, let $D_i:=F_i\cup M_i$. Then $\Delta(D_i)\le d$, and together with $F_{q-h+1},\ldots,F_q$, these give a $(q,d)$-edge coloring of $G$. Suppose that $G'$ is non-bipartite. Since $g_0(G')\ge g_0(G)$, Theorem~\ref{gold} gives
\begin{align*}
\chi'(G')&\le q-h+1+\left\lfloor\frac{q-h-2}{g_0(G')-1}\right\rfloor\\
&\le q-h+1+\left\lfloor\frac{q-h-2}{g_0(G)-1}\right\rfloor\\
&\le q-h+\left\lceil\frac{q-h}{g_0(G)-1}\right\rceil.
\end{align*}
Let
$
M_1,\ldots,M_{q-h+\lceil(q-h)/(g_0(G)-1)\rceil}
$
be the color classes of such a proper edge coloring, adding empty color classes if necessary. Since $q\le hg_0(G)$, we have
$
q-h\le h(g_0(G)-1),
$
and hence
$
\left\lceil(q-h)/({g_0(G)-1})\right\rceil\le h.
$
Therefore the number of matchings is at most $q$. Add each matching to a distinct factor $F_i$ and leave the remaining factors unchanged. These subgraphs partition $E(G)$, and every one has maximum degree at most $d$, so $\chi'_d(G)\le q$.
\end{proof}

The following lemma also follows directly from Theorem~\ref{gs}. Indeed, let $H:=\bigcup_{i=1}^sC_i$. Then $\Delta(H)=2s$. For every proper odd set $X\subsetneq V(H)$ with $|X|>2$, the underlying graph of $H[X]$ is a forest, so $|E(H[X])|\le s(|X|-1)$ and its contribution to $\Gamma_1(H)$ is at most $2s$. Every even set also contributes at most $\Delta(H)=2s$, whereas the full vertex set contributes
\[
\left\lceil\frac{|E(H)|}{\lfloor |V(H)|/2\rfloor}\right\rceil
=\left\lceil\frac{sg}{(g-1)/2}\right\rceil
=2s+\left\lceil\frac{2s}{g-1}\right\rceil\ge2s+1.
\]
Hence
$
\Gamma_1(H)=2s+\left\lceil2s/({g-1})\right\rceil,
$
and Theorem~\ref{gs} gives
\[
\chi'(H)\le\max\{\Gamma_1(H),\Delta(H)+1\}
=2s+\left\lceil\frac{2s}{g-1}\right\rceil.
\]
For completeness, we give a direct proof because its explicit decomposition will be useful.

\begin{lemma}\label{cycle-matching}
Let $g\ge3$ be odd, and let $C_1,\ldots,C_s$ be edge-disjoint copies of $C_g$ with the same underlying cycle, where $s$ is a positive integer. Then the union of $C_1,\ldots,C_s$ can be partitioned into
$
2s+\left\lceil{2s}/({g-1})\right\rceil
$
matchings.
\end{lemma}

\begin{proof}
Partition the $s$ cycles into $h:=\lceil2s/(g-1)\rceil$ groups, each containing at most $(g-1)/2$ cycles. Let one group consist of $C_1,\ldots,C_t$. Then $t\le(g-1)/2$. Let $F$ be a near-perfect matching of the underlying cycle $C_g$. Since $t\le|F|=(g-1)/2$, there are distinct edges $e_1,\ldots,e_t\in F$. For each $i\in[t]$, let $f_i$ be the edge of $C_i$ joining the same pair of vertices as $e_i$. Since $e_1,\ldots,e_t$ are pairwise disjoint, $f_1,\ldots,f_t$ are pairwise disjoint. Thus $A:=\{f_1,\ldots,f_t\}$ is a matching. For each $i\in[t]$, the graph $C_i-f_i$ is a path and can be partitioned into two matchings. Hence the union of $C_1,\ldots,C_t$ can be partitioned into $2t+1$ matchings. If the groups contain $t_1,\ldots,t_h$ cycles, then
\[
\sum_{j=1}^h(2t_j+1)=2\sum_{j=1}^ht_j+h=2s+\left\lceil\frac{2s}{g-1}\right\rceil.
\]
\end{proof}

The following lemma partitions the edges of a bounded-degree multigraph with large odd girth into matchings and a subgraph of maximum degree at most $3$.
\begin{lemma}\label{lemma:remainder}
Let $t\ge2$ be even, and let $H$ be a loopless multigraph with $\Delta(H)\le t+1$ and no odd cycle of length less than $t-1$. Then $E(H)$ can be partitioned into matchings $M_1,\ldots,M_{t-1}$ and a subgraph $D$ with $\Delta(D)\le3$.
\end{lemma}

\begin{proof}
If $t=2$, let $M_1=\emptyset$ and $D:=H$. Thus assume that $t\ge4$. Suppose that the result fails, and let $H$ be a minimum-order counterexample. If $H$ is disconnected, apply the minimality of $H$ to each component and combine the corresponding matchings and subgraphs, contradicting that $H$ is a counterexample. Thus $H$ is connected.

If $\chi'(H)\le t+2$, let $M_1,\ldots,M_{t+2}$ be the color classes of a proper edge coloring of $H$, adding empty color classes if necessary. Then $M_1,\ldots,M_{t-1}$ are matchings and
$
D:=M_t\cup M_{t+1}\cup M_{t+2}
$
has maximum degree at most $3$, contradicting that $H$ is a counterexample. Hence $\chi'(H)>t+2$. Since $\Delta(H)+1\le t+2$, Theorem~\ref{gs} implies that $\Gamma_1(H)>t+2$. An even set cannot witness this inequality, since its contribution to $\Gamma_1(H)$ is at most $\Delta(H)\le t+1$. Thus there is an odd set $X\subseteq V(H)$ with $|X|\ge3$ such that
\begin{equation}\label{eq:remainder-density}
2|E(H[X])|>(t+2)(|X|-1).
\end{equation}

If $|X|<t-1$, then $H[X]$ is bipartite, since it contains no odd cycle. Let $A\cup B$ be a bipartition of $H[X]$. Since $|X|$ is odd,
$
\min\{|A|,|B|\}\le(|X|-1)/2,
$
and hence
\begin{align*}
2|E(H[X])| &\le 2(t+1)\min\{|A|,|B|\} \\
           &\le (t+1)(|X|-1) \\
           &< (t+2)(|X|-1)
\end{align*}
contradicting \eqref{eq:remainder-density}. Thus $|X|\ge t-1$. Moreover,
$
(t+2)(|X|-1)<2|E(H[X])|\le(t+1)|X|,
$
so $|X|<t+2$. Since $|X|$ is odd and $t$ is even, $|X|\in\{t-1,t+1\}$. If $|X|=t+1$, then $(t+1)|X|=(t+1)^2$ is odd, and therefore
\[
2|E(H[X])|\le(t+1)^2-1=t(t+2)=(t+2)(|X|-1),
\]
again contradicting \eqref{eq:remainder-density}. Hence $|X|=t-1$. Consequently,
$
t^2-4<2|E(H[X])|\le t^2-1.
$
Since $2|E(H[X])|$ is even, it follows that
$
2|E(H[X])|=t^2-2.
$
Consequently,
\[
\sum_{x\in X}\bigl(t+1-d_{H[X]}(x)\bigr)=1.
\]
Thus one vertex $u\in X$ has degree $t$ in $H[X]$, and every other vertex has degree $t+1$. We first claim that $H[X]$ is non-bipartite. Otherwise, let $A\cup B$ be a bipartition of $H[X]$. Since $|X|=t-1$ is odd,
$
\min\{|A|,|B|\}\le (t-2)/2.
$
Consequently,
\[
2|E(H[X])|\le2(t+1)\min\{|A|,|B|\}\le(t+1)(t-2)<t^2-4,
\]
contradicting \eqref{eq:remainder-density}. Thus $H[X]$ contains an odd cycle. Since $|X|=t-1$ and $H$ has no odd cycle of length less than $t-1$, this cycle contains every vertex of $X$. Moreover, the underlying graph of $H[X]$ has no chord, because a chord of an odd cycle creates a shorter odd cycle. Hence
$
\underline{H[X]}\cong C_{t-1},
$
and the degree of the vertices of $H[X]$ imply that
\[
H[X]\cong R(t+1,t-1),
\]
with $u$ as its unique vertex of degree $t$. By the definition of $R(t+1,t-1)$, the graph $H[X]$ is the union of $t/2$ edge-disjoint copies $C_1,\ldots,C_{t/2}$ of $C_{t-1}$ and a near-perfect matching $N$ missing $u$. The number of edges joining $X$ to $V(H)\setminus X$ is at most
$
\sum_{x\in X}\bigl(t+1-d_{H[X]}(x)\bigr)=1.
$
Since $H$ is connected, either $X=V(H)$ or exactly one edge $e=uv$ joins $X$ to $V(H)\setminus X$. Since $u$ is the only vertex of $X$ with positive degree deficiency, this edge is incident with $u$.

Let
$
D_X:=C_1\cup N.
$
Then $d_{D_X}(u)=2$ and $d_{D_X}(x)=3$ for every $x\in X\setminus\{u\}$. By Lemma~\ref{cycle-matching}, the union of $C_2,\ldots,C_{t/2}$ can be partitioned into
\[
2\left(\frac t2-1\right)+\left\lceil\frac{t-2}{t-2}\right\rceil=t-1
\]
matchings $N_1,\ldots,N_{t-1}$. These matchings contain
$
\left( t/2-1\right)(t-1)={(t-2)(t-1)}/2
$
edges in total, and each has size at most $(t-2)/2$. Hence each $N_i$ has size $(t-2)/2$. Since every vertex has degree $t-2$ in $\bigcup_{i=1}^{t-1}N_i$, every vertex of $X$ is missed by exactly one of these matchings.

If $X=V(H)$, then $N_1,\ldots,N_{t-1},D_X$ give the required partition, contradicting that $H$ is a counterexample. Hence $X\ne V(H)$. Therefore exactly one edge $e=uv$ joins $X$ to $V(H)\setminus X$. Let $H':=H-X$. By the minimality of $H$,
\[
E(H')=E(M_1)\cup\cdots\cup E(M_{t-1})\cup E(D),
\]
where each $M_i$ is a matching and $\Delta(D)\le3$. Since $d_{H'}(v)\le t$, either $d_D(v)\le2$ or $v$ is missed by some $M_i$. If $d_D(v)\le2$, since $d_{D_X}(u)=2$, we have
$
\Delta(D\cup D_X\cup\{e\})\le3.
$
Moreover, for every $j\in[t-1]$, the union $M_j\cup N_j$ is a matching because $M_j$ lies in $H'$ and $N_j$ lies in $H[X]$, which have disjoint vertex sets. Therefore
\[
E(H)=E(D\cup D_X\cup\{e\})\cup\bigcup_{j=1}^{t-1}E(M_j\cup N_j)
\]
is the required partition, a contradiction.

Otherwise, let $M_i$ miss $v$. Since every vertex of $X$ is missed by exactly one of $N_1,\ldots,N_{t-1}$, relabel these matchings so that $N_i$ misses $u$. The matchings $M_i$ and $N_i$ lie in the vertex-disjoint graphs $H'$ and $H[X]$, respectively, and neither contains an edge incident with an endpoint of $e=uv$. Hence $M_i\cup N_i\cup\{e\}$ is a matching. For every $j\ne i$, $M_j\cup N_j$ is also a matching, and $\Delta(D\cup D_X)\le3$. Therefore
\[
E(H)=E(D\cup D_X)\cup E(M_i\cup N_i\cup\{e\})\cup\bigcup_{j\ne i}E(M_j\cup N_j)
\]
is the required partition, again a contradiction.
\end{proof}

We now use Lemma~\ref{cycle-matching} to show that the bound in Theorem~\ref{main} is sharp. Let $\Delta>d>1$ and let $g_0\ge3$, where $d$ and $g_0$ are odd. For $G:=R(\Delta,g_0)$, we have $|E(G)|=\lfloor\Delta g_0/2\rfloor$ and $\lfloor dg_0/2\rfloor=(dg_0-1)/2$. Hence
\[
\Gamma_d(G)\ge\left\lceil\frac{|E(G)|}{\lfloor dg_0/2\rfloor}\right\rceil
=\left\lceil\frac{\Delta g_0-1}{dg_0-1}\right\rceil=p(G),
\]
where, when $\Delta$ is even, the equality follows because $(\Delta g_0-1)/(dg_0-1)$ is not an integer.

\begin{lemma}\label{ring-coloring}
For odd integers $d>1$ and $g_0\ge3$ and every $\Delta>d$,
\[
\chi'_d(R(\Delta,g_0))=\left\lceil\frac{\Delta g_0-1}{dg_0-1}\right\rceil.
\]
\end{lemma}

\begin{proof}
Let $G:=R(\Delta,g_0)$, let
$
p:=\lceil(\Delta g_0-1)/(dg_0-1)\rceil,
$
and let $b:=(d-1)/2$. Suppose first that $\Delta$ is odd. Then $G$ is the union of $m:=(\Delta-1)/2$ edge-disjoint copies of $C_{g_0}$ and a near-perfect matching $M^*$. If $m\le pb$, distribute these cycles among $p$ subgraphs, each containing at most $b$ cycles, and add $M^*$ to one of them. This gives a $(p,d)$-edge coloring. Otherwise, let $H_1,\ldots,H_p$ each consist of $b$ cycles. The remaining graph consists of $M^*$ and
$
r:=(\Delta-1-p(d-1))/2
$
cycles. By Lemma~\ref{cycle-matching}, the $r$ remaining cycles can be partitioned into
$
2r+\left\lceil{2r}/({g_0-1})\right\rceil
$
matchings. Together with $M^*$, this gives
$
q:=2r+\left\lceil{2r}/({g_0-1})\right\rceil+1
$
matchings. Since
\[
q-p=\left\lceil\frac{\Delta g_0-1-p(dg_0-1)}{g_0-1}\right\rceil\le0,
\]
we have $q\le p$. Add empty matchings if $q<p$, and add the resulting $p$ matchings to $H_1,\ldots,H_p$. This gives a $(p,d)$-edge coloring.

Now suppose that $\Delta$ is even. Then $G$ is the union of $m:=\Delta/2$ edge-disjoint copies of $C_{g_0}$. If $m\le pb$, distribute the cycles among $p$ subgraphs as above. Otherwise, after forming $H_1,\ldots,H_p$, the remaining graph consists of
$
r:=(\Delta-p(d-1))/2
$
cycles and can be partitioned into
$
q:=2r+\left\lceil{2r}/({g_0-1})\right\rceil
$
matchings. By the definition of $p$, the integer $\Delta g_0-p(dg_0-1)$ is at most $1$. It is even because both $\Delta g_0$ and $p(dg_0-1)$ are even, and hence it is at most $0$. Therefore
\[
q-p=\left\lceil\frac{\Delta g_0-p(dg_0-1)}{g_0-1}\right\rceil\le0.
\]
Thus $q\le p$. Add empty matchings if $q<p$, and add the resulting $p$ matchings to $H_1,\ldots,H_p$. Hence $\chi'_d(G)\le p$.

By the preceding calculation, $\Gamma_d(G)\ge p(G)=p$. Since $\chi'_d(G)\ge\Gamma_d(G)$ and we have proved that $\chi'_d(G)\le p$, we obtain
$
p\le\Gamma_d(G)\le\chi'_d(G)\le p.
$
Therefore
$
\chi'_d(G)=\Gamma_d(G)=p.
$
\end{proof}

Thus the value $p(G)$ in Theorem~\ref{main} is best possible.

\section{Proof of the main result}\label{main-result}

\begin{proof}[Proof of Theorem~\ref{main}]
The case $d=1$ follows from Theorem~\ref{gold}. Hence we may assume that $d\ge3$. If $p(G)=1$, then $g_0(G)\Delta(G)-1\le dg_0(G)-1$, so $\Delta(G)\le d$ and $\chi'_d(G)=1$. Thus assume that $p(G)\ge2$.

Suppose that the theorem is false. Choose the smallest odd integer $d\ge3$ for which a counterexample exists, and let
\[
c:=\min\left\{\left\lceil\frac{p(H)-1}{g_0(H)}\right\rceil:H\text{ is a counterexample for }d\right\}.
\]
Choose a counterexample $G$ attaining this minimum.

First, we show that we can assume that a counterexample is regular without changing its odd girth. We then prove that it satisfies $g_0(G)\ge5$, $\Delta(G)=p(G)d-c$, and $p(G)-c$ is odd. We next prove that a minimum-order counterexample has at most one cut-edge and contains the required even factors, and reduce to $d=3$, $c=1$, and $\Delta(G)=3p(G)-1$, where $p(G)\ge2$ is even and $g_0(G)\ge p(G)-1$. Finally, we combine $p(G)-1$ edge-disjoint $2$-factors of $G$ with a partition of the remaining edges into $p(G)-1$ matchings and a subgraph of maximum degree at most $3$.

\begin{claim}\label{claim:regularization}
For every integer $r\ge\Delta(G)$, there is an $r$-regular loopless multigraph $\widehat G$ containing a copy of $G$ such that $g_0(\widehat G)=g_0(G)$.
\end{claim}

\begin{proof}[Proof of Claim~\ref{claim:regularization}]
Let $G_1$ and $G_2$ be two disjoint copies of $G$. For every $v\in V(G)$, add $r-d_G(v)$ edges between the copies of $v$ in $G_1$ and $G_2$, and let $\widehat G$ be the resulting multigraph. Then $\widehat G$ is $r$-regular and contains a copy of $G$, so $g_0(\widehat G)\le g_0(G)$. Suppose that $\widehat G$ has an odd cycle $C$ of length less than $g_0(G)$. The cycle $C$ is not contained in $G_1$ or $G_2$, so it uses added edges. Let $f_1,\ldots,f_t$ be these edges in cyclic order. Since each added edge switches between $G_1$ and $G_2$, $t$ is even. Delete $f_1,\ldots,f_t$ from $C$ and identify corresponding vertices and edges of $G_1$ and $G_2$. The remaining edges form an odd closed walk in $G$ of length $|C|-t<g_0(G)$ and hence contain a shorter odd cycle, a contradiction. Therefore $g_0(\widehat G)=g_0(G)$.
\end{proof}

Apply the claim with $r=\Delta(G)$. Then $\widehat G$ has the same values of $d$, $\Delta(G)$, $g_0(G)$, $p(G)$, and $c$ as $G$. Moreover, a $(p(G),d)$-edge coloring of $\widehat G$ would restrict to one of $G$, so $\widehat G$ is also a counterexample. Hence we may assume that $G$ is $\Delta(G)$-regular and among all regular counterexamples with the chosen $d$ and $c$, choose $G$ of minimum order, and let
$
p:=p(G)$, $ g_0:=g_0(G)$ and $\Delta:=\Delta(G).
$

\begin{claim}\label{claim:main-reduction}
We have $g_0\ge5$, $\Delta=pd-c$, and $p-c$ is odd.
\end{claim}

\begin{proof}[Proof of Claim~\ref{claim:main-reduction}]
If $g_0=3$, then Theorem~\ref{3french} gives
$
\chi'_d(G)\le\left\lceil({3\Delta-1})/({3d-1})\right\rceil=p,
$
a contradiction. Thus $g_0\ge5$. By the choice of $p$, we have $
g_0\Delta-1\le p(dg_0-1)=pdg_0-p
$
and since $\Delta$ is an integer, we have
$
\Delta\le pd-\left\lceil({p-1})/{g_0}\right\rceil=pd-c.
$

Suppose that $p-c$ is even. Then, since $d$ is odd, $pd-c$ is even as well. Since $c=\lceil(p-1)/g_0\rceil$, we have $p\le cg_0+1$. If $p=cg_0+1$, then
$
p-c=c(g_0-1)+1
$
is odd, a contradiction. Hence $p\le cg_0$ and also by the choice of $c$, we have $c\le p$. By Claim~\ref{claim:regularization}, applied with $r=pd-c$, there is a $(pd-c)$-regular loopless multigraph $\widehat G$ containing a copy of $G$ such that $g_0(\widehat G)=g_0$. Since $c\le p\le cg_0$ and $pd-c$ is even, Lemma~\ref{lemma:absorption}, applied with $q=p$ and $h=c$, gives $\chi'_d(\widehat G)\le p$. Restricting this coloring to $G$ gives $\chi'_d(G)\le p$, a contradiction. Therefore $p-c$ is odd.

Suppose that $\Delta<pd-c$. Then
$
\Delta\le pd-(c+1).
$
Moreover, since $d$ and $p-c$ are both odd, $pd-(c+1)$ is even. Again, by the choice of $c$, we have 
$
c+1\le p\le cg_0+1\le(c+1)g_0.
$
By Claim~\ref{claim:regularization}, applied with $r=pd-(c+1)$, there is a $(pd-(c+1))$-regular loopless multigraph $\widehat G$ containing a copy of $G$ such that $g_0(\widehat G)=g_0$. Lemma~\ref{lemma:absorption}, applied with $q=p$ and $h=c+1$, gives $\chi'_d(\widehat G)\le p$. Restricting this coloring to $G$ gives $\chi'_d(G)\le p$, again a contradiction. Therefore $\Delta=pd-c$.
\end{proof}

\begin{claim}\label{claim:smaller}
Every $\Delta$-regular loopless multigraph $G'$ with $|V(G')|<|V(G)|$ that is bipartite or has odd girth at least $g_0$ is $(p,d)$-edge colorable.
\end{claim}

\begin{proof}[Proof of Claim~\ref{claim:smaller}]
If $G'$ is bipartite, then Theorem~\ref{konig}, followed by grouping the matching color classes into groups of at most $d$, gives
$
\chi'_d(G')\le\left\lceil{\Delta}/{d}\right\rceil\le p.
$
Suppose that $G'$ is non-bipartite, and let $h:=g_0(G')\ge g_0$ and $p':=p(G')$. Since $h\ge g_0$ and $\Delta>d$,
\[
\frac{g_0\Delta-1}{dg_0-1}-\frac{h\Delta-1}{dh-1}
=\frac{(h-g_0)(\Delta-d)}{(dg_0-1)(dh-1)}\ge0,
\]
and hence $p'\le p$. Therefore
$
\left\lceil({p'-1})/{h}\right\rceil\le c.
$
If $G'$ were not $(p,d)$-edge colorable, then it would not be $(p',d)$-edge colorable. By the minimality of $c$, we would have
$
\lceil(p'-1)/h\rceil=c.
$
The argument of Claim~\ref{claim:main-reduction}, applied to $G'$, would then give
$
\Delta=p'd-c.
$
Since $\Delta=pd-c$, we would have $p'=p$, contradicting the choice of $G$ as a minimum-order regular counterexample with the chosen $d$ and $c$. Hence $G'$ is $(p,d)$-edge colorable.
\end{proof}

In particular, $G$ is connected, since otherwise its components could be colored separately.

\begin{claim}\label{claim:few-bridges}
$G$ has at most one cut-edge.
\end{claim}

\begin{proof}[Proof of Claim~\ref{claim:few-bridges}] Since $d$ is odd and $p-c$ is odd, $\Delta=pd-c$ is odd. Suppose that $e=uv$ is a cut-edge, and let $A$ and $B$ be the components of $G-e$ containing $u$ and $v$, respectively. In $A$, the vertex $u$ has degree $\Delta-1$ and every other vertex has degree $\Delta$. Hence
$
\Delta|V(A)|-1=2|E(A)|,
$
so $|V(A)|$ is odd. Moreover, $A$ contains an odd cycle. Otherwise, let $X\cup Y$ be a bipartition of $A$ with $u\in X$. Then
$
\Delta|X|-1=|E(A)|=\Delta|Y|,
$
which is impossible. Thus $|V(A)|\ge g_0$. The same holds for $B$.

First, assume that both $|V(A)|>g_0$ and $|V(B)|>g_0$. Take a copy of $R(\Delta,g_0)$, let $w$ be its unique vertex of degree $\Delta-1$, and add the edge $uw$. The resulting $\Delta$-regular multigraph has no odd cycle shorter than $g_0$ and has fewer vertices than $G$, so it is $(p,d)$-edge colorable by Claim~\ref{claim:smaller}. Construct the analogous graph from $B$. Restrict the two colorings to $A$ and $B$, permute the colors on one side so that the two deleted joining edges have the same color, and give $e$ this color. This gives a $(p,d)$-edge coloring of $G$, a contradiction. Therefore at least one of $A$ and $B$ has order $g_0$. By symmetry, assume that $|V(A)|=g_0$.

We know $A$ contains an odd cycle of length at least $g_0$. Let 
$
C=v_1v_2\cdots v_{g_0}v_1
$
be that cycle with $u=v_1$. The cycle $C$ has no chord in $\underline A$, since every chord of an odd cycle creates an odd cycle of length less than $g_0$. Thus $\underline A\cong C_{g_0}$. For $i\in[g_0]$, let $m_i:=\mu(v_iv_{i+1})$, where the indices are taken modulo $g_0$. Since $d_A(v_i)=\Delta$ for $2\le i\le g_0$, we have
$
m_{i-1}+m_i=\Delta.
$
Hence the multiplicities alternate between $m_1$ and $\Delta-m_1$. Since $g_0$ is odd, $m_{g_0}=m_1$. Also, $d_A(u)=\Delta-1$, so
$
2m_1=m_{g_0}+m_1=\Delta-1.
$
Therefore $m_1=m_{g_0}=\lfloor\Delta/2\rfloor$, and the remaining multiplicities alternate between $\lceil\Delta/2\rceil$ and $\lfloor\Delta/2\rfloor$. Hence $A\cong R(\Delta,g_0)$ with $u$ as its unique vertex of degree $\Delta-1$. Therefore every cut-edge of $G$ has a copy of $R(\Delta,g_0)$ on one side.

Suppose that $G$ has at least two cut-edges. In the tree whose vertices are the components obtained by deleting all cut-edges, choose two leaf components. For a cut-edge incident with a leaf component, the opposite side contains another cut-edge because the tree has at least two edges. Hence the opposite side cannot be an almost full ring graph, so the leaf component is the ring side. Therefore both chosen leaf components are copies of $R(\Delta,g_0)$. Delete them and their incident cut-edges, and let $u,u'$ be the two vertices in the remaining graph that were incident with these cut-edges. Add vertices $x_0,\ldots,x_{g_0}$, add the edges $ux_0$ and $x_{g_0}u'$ with multiplicity one, and let
\[
\mu(x_ix_{i+1})=
\begin{cases}
\Delta-1,&\text{if }i\text{ is even},\\
1,&\text{if }i\text{ is odd},
\end{cases}
\qquad 0\le i\le g_0-1.
\]
The resulting multigraph $G'$ is $\Delta$-regular. Any odd cycle containing a new internal vertex must contain the entire added $u$--$u'$ path, which has length $g_0+2$. Hence every new odd cycle has length at least $g_0+2$, including when $u=u'$, so $G'$ has no odd cycle shorter than $g_0$. Moreover, $G'$ has fewer vertices than $G$. Hence $G'$ is $(p,d)$-edge colorable by Claim~\ref{claim:smaller}. Delete the added path. Let the colors of $ux_0$ and $x_{g_0}u'$ be $\alpha$ and $\beta$, respectively. Since $R(\Delta,g_0)$ has maximum degree $\Delta$ and odd girth $g_0$, we have
$
p(R(\Delta,g_0))=p.
$
By Lemma~\ref{ring-coloring}, $R(\Delta,g_0)$ has a $(p,d)$-edge coloring. Let $w$ be the unique vertex of degree $\Delta-1$ in $R(\Delta,g_0)$. In a $(p,d)$-edge coloring of $R(\Delta,g_0)$, the sum of the degrees of the $p$ color classes at $w$ is $\Delta-1$. Since their total capacity at $w$ is $pd$, we have
$
pd-(\Delta-1)=c+1>0.
$
Hence some color has degree at most $d-1$ at $w$. By permuting the colors, we may assume that this color is any prescribed color, so the cut-edge incident with $w$ can be assigned that color without making its degree exceed $d$. Restore the two deleted copies of $R(\Delta,g_0)$ and assign their incident cut-edges the colors $\alpha$ and $\beta$, respectively. If $u=u'$ and $\alpha=\beta$, deleting the two path edges of this color leaves color degree at most $d-2$ at $u$, so restoring both cut-edges still gives color degree at most $d$. Thus the resulting coloring is a $(p,d)$-edge coloring of $G$, a contradiction. Thus $G$ has at most one cut-edge.
\end{proof}

\begin{claim}\label{claim:even-factor}
$G$ has a $k$-factor for every positive even integer $k\le2\Delta/3$.
\end{claim}

\begin{proof}[Proof of Claim~\ref{claim:even-factor}]
Let $k\le2\Delta/3$ be a positive even integer. If $G$ is $2$-edge-connected, then the result follows from Theorem~\ref{kano}. Thus suppose that $G$ has a cut-edge $u_1v_1$. Let $X$ and $Y$ be the components of $G-u_1v_1$ containing $u_1$ and $v_1$, respectively. By Claim~\ref{claim:few-bridges}, $u_1v_1$ is the unique cut-edge of $G$, so both $X$ and $Y$ are $2$-edge-connected. By the proof of Claim~\ref{claim:few-bridges}, we may assume that
$
Y\cong R(\Delta,g_0).
$
Let the underlying cycle of $Y$ be formed by
$
v_1v_2\cdots v_{g_0}v_1.
$
Choose an edge $e_X=u_1w\in E(X)$, and let $e_Y$ be an edge joining $v_1$ and $v_2$. Delete $e_X$ and $e_Y$, and add one additional edge joining $u_1$ and $v_1$ and one edge joining $v_2$ and $w$. Let $G'$ be the resulting multigraph. Then $G'$ is $\Delta$-regular. Since $X$ and $Y$ are $2$-edge-connected, both $X-e_X$ and $Y-e_Y$ are connected, and hence $G'$ is connected. The two edges joining $u_1$ and $v_1$ and the edge $v_2w$ are not cut-edges, since after deleting any one of them at least one of the other two still joins $X-e_X$ to $Y-e_Y$. Suppose that $f\in E(X-e_X)$ is a cut-edge of $X-e_X$. Since $f$ is not a cut-edge of $X$, the vertices $u_1$ and $w$ lie in different components of $(X-e_X)-f$. In $G'-f$, however, they are joined by a path consisting of an edge from $u_1$ to $v_1$, a $v_1$--$v_2$ path in $Y-e_Y$, and the edge $v_2w$. Thus $f$ is not a cut-edge of $G'$. The same argument applies to every cut-edge of $Y-e_Y$. Therefore $G'$ is $2$-edge-connected. By Theorem~\ref{kano}, $G'$ has a $k$-factor $F$ containing $v_2w$. Let
$
x:=\mu_F(v_1v_2)$ and $y:=\mu_F(v_1v_{g_0}).
$
Since $G'$ has exactly two edges joining $u_1$ and $v_1$,
$
x+y\ge k-2.
$
Since $v_2w\in E(F)$, the degree equations along the cycle give
$
\mu_F(v_{2i}v_{2i+1})=k-x-1
$
and
$
\mu_F(v_{2i+1}v_{2i+2})=x+1
$
for $i=1,\ldots,(g_0-1)/2$, where the indices are taken modulo $g_0$. In particular,
$
y=x+1.
$
Therefore $x+y$ is odd. Since
$
k-2\le x+y\le k
$
and $k$ is even, we have
$
x+y=k-1.
$
Thus $F$ contains exactly one of the two edges joining $u_1$ and $v_1$. Delete this edge and $v_2w$ from $F$, and restore the deleted edges $u_1w$ and $v_1v_2$. The resulting spanning subgraph is a $k$-factor of $G$.
\end{proof}

\begin{claim}\label{claim:d-three}
$d=3$.
\end{claim}

\begin{proof}[Proof of Claim~\ref{claim:d-three}]
Suppose that $d\ge5$. By the definition of $c$, we have $c\le p-1$, and Claim~\ref{claim:main-reduction} gives $\Delta=pd-c$. Hence
$
2\Delta-6p=2p(d-3)-2c\ge4p-2c>0.
$
This implies $2p<2\Delta/3$, and so Claim~\ref{claim:even-factor} gives a $2p$-factor $F$ in $G$. By repeatedly applying Theorem~\ref{pet}, let
$
F=F_1\cup\cdots\cup F_p,
$
where each $F_i$ is a $2$-factor. Let $H:=G-F$. Then $H$ is $(p(d-2)-c)$-regular. We show that $\chi'_{d-2}(H)\le p.$

If $H$ is bipartite, then Theorem~\ref{konig}, followed by grouping the matching color classes into groups of at most $d-2$, gives
\[
\chi'_{d-2}(H)\le\left\lceil\frac{p(d-2)-c}{d-2}\right\rceil\le p.
\]
Otherwise, let $h:=g_0(H)\ge g_0$ and
$
p':=\left\lceil(h\Delta(H)-1)/({(d-2)h-1})\right\rceil.
$
Since by the choice of $c$, we have $p\le cg_0+1\le ch+1$,
\[
h\Delta(H)-1=h\bigl(p(d-2)-c\bigr)-1\le p\bigl((d-2)h-1\bigr).
\]
Hence $p'\le p$. Since $d-2<d$ is odd, the minimality of $d$ gives
$
\chi'_{d-2}(H)\le p'\le p.
$

Thus, in either case, let $H_1,\ldots,H_p$ be the color classes of a $(p,d-2)$-edge coloring of $H$, adding empty color classes if necessary. For every $i\in[p]$, let $D_i:=H_i\cup F_i$. Then $\Delta(D_i)\le(d-2)+2=d$, so $\chi'_d(G)\le p$, a contradiction. Thus $d=3$.
\end{proof}
\begin{claim}\label{claim:c-one}
$c=1$.
\end{claim}

\begin{proof}[Proof of Claim~\ref{claim:c-one}]
By Claim~\ref{claim:d-three}, $d=3$. Suppose that $c\ge2$. Since $c=\lceil(p-1)/g_0\rceil$, the identity $\lceil a/b\rceil=1+\lfloor(a-1)/b\rfloor$ for positive integers $a$ and $b$ implies that there is an integer $r$ with $2\le r\le g_0+1$ such that
$
p=(c-1)g_0+r.
$
Since
$
p-c=(c-1)(g_0-1)+r-1
$
is odd and $g_0-1$ is even, $r$ is even. By Claim~\ref{claim:main-reduction}, $\Delta=3p-c$, and since $p=(c-1)g_0+r$ and $2\le r\le g_0+1$, we have
$
2\Delta-3(3r-4)\ge3g_0+5>0.
$
Thus $3r-4<2\Delta/3$, and since $3r-4$ is even, Claim~\ref{claim:even-factor} gives a $(3r-4)$-factor $F$ in $G$. We first show that $\chi'_3(F)\le r-1$.

If $r=2$, then $F$ is a $2$-factor and $\chi'_3(F)=1=r-1$. Suppose that $r\ge4$. If $F$ is bipartite, then Theorem~\ref{konig}, followed by grouping the matching color classes into groups of at most three, gives
$
\chi'_3(F)\le\left\lceil({3r-4})/{3}\right\rceil=r-1.
$
Suppose that $F$ is non-bipartite. Since $1<r-1\le g_0\le g_0(F)$ and $3(r-1)-1$ is even, Lemma~\ref{lemma:absorption}, applied with $d=3$, $q=r-1$, and $h=1$, gives $\chi'_3(F)\le r-1$.

Let $H:=G-F$. Then $H$ is regular,
$
\Delta(H)=3(p-r+1)-(c-1),
$
and since $p=(c-1)g_0+r$, we have
$
p-r+1=(c-1)g_0+1.
$
We now show that $\chi'_3(H)\le p-r+1$.

If $H$ is bipartite, then Theorem~\ref{konig}, followed by grouping the matching color classes into groups of at most three, gives
$
\chi'_3(H)\le\left\lceil{\Delta(H)}/3\right\rceil\le p-r+1.
$
Suppose that $H$ is non-bipartite, and let $h:=g_0(H)\ge g_0$. Since
$
p-r+1=(c-1)g_0+1\le(c-1)h+1,
$
we have
\[
h\Delta(H)-1=h\bigl(3(p-r+1)-(c-1)\bigr)-1\le(p-r+1)(3h-1),
\]
and hence $p(H)\le p-r+1$. Moreover,
\[
\left\lceil\frac{p(H)-1}{h}\right\rceil\le\left\lceil\frac{p-r}{h}\right\rceil=\left\lceil\frac{(c-1)g_0}{h}\right\rceil\le c-1.
\]
If $\chi'_3(H)>p-r+1$, then $\chi'_3(H)>p(H)$, so $H$ is a counterexample with defect $3$ and
$
\left\lceil(p(H)-1)/{g_0(H)}\right\rceil<c,
$
contradicting the minimality of $c$. Hence $\chi'_3(H)\le p-r+1$.

Using disjoint color sets for $F$ and $H$ gives
$
\chi'_3(G)\le(r-1)+(p-r+1)=p,
$
a contradiction. Thus $c=1$.
\end{proof}

Now we complete the proof. By Claims~\ref{claim:d-three} and~\ref{claim:c-one}, we have
$
d=3$, $\Delta=3p-1$ and $c=\left\lceil({p-1})/{g_0}\right\rceil=1.$
Since $p\ge2$, we have
$
2\Delta-3(2p-2)=4>0,
$
and so $2p-2<2\Delta/3$. Hence Claim~\ref{claim:even-factor} gives a $(2p-2)$-factor $F$ in $G$. By repeatedly applying Theorem~\ref{pet}, let
$
F=F_1\cup\cdots\cup F_{p-1},
$
where each $F_i$ is a $2$-factor. Let $H:=G-F$. Then $H$ is $(p+1)$-regular. Since $p-c=p-1$ is odd by Claim~\ref{claim:main-reduction}, $p$ is even. Moreover, $\lceil(p-1)/g_0\rceil=1$ implies $g_0\ge p-1$, so $H$ has no odd cycle of length less than $p-1$. Applying Lemma~\ref{lemma:remainder} with $t=p$, we obtain a partition
\[
E(H)=E(M_1)\cup\cdots\cup E(M_{p-1})\cup E(D),
\]
where each $M_i$ is a matching and $\Delta(D)\le3$. For every $i\in[p-1]$, let
$
D_i:=F_i\cup M_i.
$
Then $\Delta(D_i)\le3$, and
\[
E(G)=E(D_1)\cup\cdots\cup E(D_{p-1})\cup E(D).
\]
Hence $G$ is $(p,3)$-edge colorable, contradicting that $G$ is a counterexample.
\end{proof}

As an application of Theorem~\ref{main}, we obtain the following case of Conjecture~\ref{dgs}.

\begin{corollary}\label{cor:dgs-concrete}
Let $G$ be a loopless non-bipartite multigraph, let $d>1$ be odd, and let $\Delta(G)=md+s$, where $m\ge0$ and $0\le s<d$. Let
$
\ell:=\max\{1,\Gamma_d(G)-m\}.
$
Then $\ell d-s>0$. If
\[
g_0(G)\ge\left\lceil\frac{m+\ell-1}{\ell d-s}\right\rceil,
\]
then
$
\chi'_d(G)\le m+\ell,
$
and hence Conjecture~\ref{dgs} holds. 
\end{corollary}

\begin{proof}
Let
$
K:=\max\{\Gamma_d(G),\lceil(\Delta(G)+1)/d\rceil\}.
$
Since $\lceil(\Delta(G)+1)/d\rceil=m+1$, we have $K=m+\ell$ and
$
Kd-\Delta(G)=\ell d-s.
$
The stated odd-girth condition gives
\[
K-1=m+\ell-1\le g_0(G)(\ell d-s)=g_0(G)(Kd-\Delta(G)),
\]
and hence
$
g_0(G)\Delta(G)-1\le K(dg_0(G)-1).
$
Theorem~\ref{main} gives $\chi'_d(G)\le K=m+\ell$. 
\end{proof}

\end{document}